\documentclass[psamsfonts, 10pt]{amsart}

\usepackage{amssymb}
\usepackage{amsmath}
\usepackage{amsthm}
\usepackage[urlcolor=black, colorlinks=true, citecolor=black, linkcolor=black]{hyperref}
\usepackage[english]{babel}
\usepackage{graphics}
\usepackage{enumerate}
\usepackage[list=off]{caption}

\usepackage{dsfont} 

\usepackage{mathtools}

\usepackage{tikz}
\tikzset{node distance=3cm, auto}
\usetikzlibrary{calc}
\usetikzlibrary{patterns}

\newcommand{ \R } { \mathbb{R} }
\newcommand{ \N } { \mathbb{N} }
\newcommand{ \Z } { \mathbb{Z} }

\newcommand{\w}{\omega}

\newcommand{\cont}{\mathfrak{c}}

\newcommand{\script}{\mathcal}
\newcommand{\parentheses}[1]{{\left( {#1} \right)}}
\newcommand{\sequence}[1]{{\langle {#1} \rangle}}
\def\Sequence#1:#2{\sequence{{#1} \colon {#2}}}
\newcommand{\p}{\parentheses}
\newcommand{\of}{\parentheses}
\newcommand{\closure}[1]{\overline{#1}}
\newcommand{\closureIn}[2]{\closure{#1}^{#2}}
\newcommand{\interior}[1]{\mathrm{int}\of{#1}}

\newcommand{\Set}[1]{{\left\lbrace {#1} \right\rbrace}}
\newcommand{\singleton}{\Set}
\newcommand{\union}{\cup}
\newcommand{\intersect}{\cap}
\newcommand{\Union}{\bigcup}

\newcommand{\cardinality}[1]{{\left\lvert {#1} \right\rvert}}

\def\set#1:#2{\Set{{#1} \colon {#2}}}

\newcommand{\diam}[1]{\textnormal{diam}{\left({#1} \right)}}

\newcommand{\singletonDeletion}[1]{\script{D}\parentheses{#1}}

\newcommand{\Intersection}{\bigcap}

\newcommand{\concat}{%
  \mathord{
    \mathchoice
    {\raisebox{1ex}{\scalebox{.7}{$\frown$}}}
    {\raisebox{1ex}{\scalebox{.7}{$\frown$}}}
    {\raisebox{.7ex}{\scalebox{.5}{$\frown$}}}
    {\raisebox{.7ex}{\scalebox{.5}{$\frown$}}}
  }
}
\newcommand{\relativization}[2]{\left.{#1}\right|\!{ {#2}}}

\begin{document}
\title{Reconstructing Compact Metrizable Spaces}
\author[Gartside, Pitz, Suabedissen]{Paul Gartside, Max F.\ Pitz, Rolf Suabedissen}
\thanks{This research formed part of the second author's thesis at the University of Oxford \cite{thesis}, supported by an EPSRC studentship.}
\address{The Dietrich School of Arts and Sciences\\301 Thackeray Hall\\Pittsburgh PA 15260}
\email{gartside@math.pitt.edu}
\address{Mathematical Institute\\University of Oxford\\Oxford OX2 6GG\\United Kingdom}
\email[Corresponding author]{max.f.pitz@gmail.com}
\email{rolf.suabedissen@maths.ox.ac.uk}

\begin{abstract}
The deck, $\singletonDeletion{X}$, of a topological space $X$ is the set $\singletonDeletion{X}=\set{[X \setminus \singleton{x}]}:{x \in X}$, where $[Y]$ denotes the homeomorphism class of $Y$. A space $X$ is (topologically) \emph{reconstructible}  if whenever $\singletonDeletion{Z}=\singletonDeletion{X}$ then $Z$ is homeomorphic to $X$. It is known that every (metrizable) continuum is reconstructible, whereas the Cantor set is non-reconstructible.

The main result of this paper characterises the non-reconstructible compact metrizable spaces as precisely those where for each point $x$ there is a sequence $\Sequence{B_n^x}:{n \in \N}$ of pairwise disjoint clopen subsets converging to $x$ such that $B_n^x$ and $B_n^y$ are homeomorphic for each $n$, and all $x$ and $y$. 

In a non-reconstructible compact metrizable space the set of $1$-point components forms a  dense $G_\delta$. For $h$-homogeneous spaces, this condition is sufficient for non-reconstruction. A wide variety of spaces with a dense $G_\delta$ set of $1$-point components are presented, some reconstructible and others not reconstructible. 
\end{abstract}

\maketitle
\thispagestyle{plain}

\newtheorem*{recresult}{Theorem~\ref{fincompactrec}}
\newtheorem{mythm}{Theorem} \numberwithin{mythm}{section} 
\newtheorem{myprop}[mythm]{Proposition}
\newtheorem{myobs}[mythm]{Observation}
\newtheorem{mycor}[mythm]{Corollary}
\newtheorem{mylem}[mythm]{Lemma} 
\newtheorem{myquest}[mythm]{Question} 
\newtheorem{myprob}[mythm]{Problem}
\newtheorem*{myconj}{Conjecture}
\newtheorem{mydef}[mythm]{Definition}
\newtheorem{myclaim}[mythm]{Claim}
\newtheorem{myremark}[mythm]{Remark}
\newtheorem{mycase}{Case}
\newtheorem{mycase2}{Case}
\newtheorem*{myclmn}{Claim}

\section{Introduction}
The \emph{deck} of a  graph $G$ is the set $\singletonDeletion{G}=\set{[G-x]}:{x \in G}$, where $[G-x]$ is  the isomorphism class of the graph obtained from $G$ by deleting the vertex $x$, and all incident edges. 
Then a graph $G$ is  \emph{reconstructible}  if whenever $\singletonDeletion{G}=\singletonDeletion{H}$ then  $G$ is isomorphic to $H$.
Kelly and Ulam's well-known \emph{Graph Reconstruction Conjecture} from 1941 proposes that every finite graph with at least three vertices is reconstructible. For more information, see for example \cite{Bondy}. 

Similarly, the \emph{deck} of a topological space $X$ is the set $\singletonDeletion{X}=\set{[X \setminus \singleton{x}]}:{x \in X}$, where $[X \setminus \{x\}]$ denotes the homeomorphism class of $X \setminus \{x\}$. Any space $Y$ homeomorphic to some $X \setminus \{x\}$ is called a \emph{card} of $X$. A space $Z$ is a \emph{reconstruction} of $X$ if $Z$ has the same cards as $X$, i.e.\ $\singletonDeletion{Z}=\singletonDeletion{X}$. Further, $X$ is \emph{topologically reconstructible} if whenever $\singletonDeletion{Z}=\singletonDeletion{X}$ then $Z$ and $X$ are homeomorphic.

Many familiar spaces such as $I=[0,1]$, Euclidean $n$-space, $n$-spheres, the rationals and the irrationals are reconstructible, as are compact spaces containing an isolated point, \cite{recpaper}. However the  Cantor set $C$ and  $C \setminus \{0\}$ have the same deck, $\{C \setminus \{0\}\}$. Hence the Cantor set is an example of a non-reconstructible compact, metrizable space. Note that $C \setminus \{0\}$ is non-compact. Indeed, \cite[5.2]{recpaper} states that if all reconstructions of a compact space are compact then it is reconstructible. 

In the other direction, the authors proved that every continuum, i.e.\ every connected compact metrizable space, is reconstructible, \cite[4.12]{part1}. The proof proceeds in two steps. First, every continuum has cards  with a maximal finite compactification. Second, any compact space that has a card with a maximal finite compactification has only compact reconstructions---and so it is reconstructible by the result mentioned above, \cite[3.11]{part1}.

We start this paper by showing that a compact metrizable space $X$ without isolated points has a card with a maximal finite compactification if and only if  $C_1(X) = \set{x \in X}:{ \singleton{x} \textnormal{ is a 1-point component of }X}$ is \emph{not} a dense $G_\delta$ subset of $X$ (Theorem~\ref{charfinitecompactificationcptmetr}). Hence, every non-reconstructible compact metrizable space contains a dense $G_\delta$ of $1$-point components (Theorem~\ref{lotsofreconstrutible}). However, we will also give examples of compact metrizable spaces which are reconstructible despite having a dense set of $1$-point components (Lemma~\ref{C1dense_butrecon}). Hence, our approach for continua---looking for cards with maximal finite compactifications---does not work in general. 


Instead, we characterise when a compact metrizable space is reconstructible via \lq universal sequences\rq. A space $X$ is said to have a \emph{universal sequence} if for each point $x$ there is a sequence   $\Sequence{B_n^x}:{n \in \N}$ of pairwise disjoint clopen subsets converging to $x$ such that $B_n^x$ and $B_n^y$ are homeomorphic for every $x,y\in X$ and for each $n \in \N$. 
\begin{mythm}[Reconstruction Characterisation]\label{recchartthm}
A compact metrizable space is reconstructible if and only if it does not have a universal sequence.
\end{mythm}
The reverse implication is the content of Proposition~\ref{universalsequence}, while the forward implication is Proposition~\ref{recchartlemma}.

Theorems~\ref{stronghomogeneousnonrec} and~\ref{tricktrick} connect our characterisation of non-reconstructible compact metrizable spaces with the existence of certain types of dense subsets of $1$-point components. If $X$ is $h$-homogeneous, i.e.\ all non-empty clopen subsets of $X$ are homeomorphic to $X$, then all that is required for $X$ to be non-reconstructible is that $C_1(X)$ is dense. 
And  if $C_1(X)$ has non-empty interior then $X$ is non-reconstructible if and only if the interior of $C_1(X)$ has no isolated points and is dense  in $X$, which  holds if and only if $X$ is a compactification of $C \setminus \{0\}$. It follows, see Section~\ref{section61} for details and more related results, that there are many compact, metrizable non-reconstructible spaces of this type.

For compact metrizable spaces where $C_1(X)$ is both dense and co-dense, the situation is considerably more complex. We provide two methods of constructing such spaces which are $h$-homogeneous and non-reconstructible (Corollary~\ref{firstctblehhom} and Theorem~\ref{constructionhhomo1}). Whereas non-reconstructible $h$-homogeneous spaces have a universal sequence with all terms homeomorphic to each other (and to $X$), we also construct examples without such a `constant' universal sequence (Theorems~\ref{noconseq1} and~\ref{noconseq2}). 

In light of these results there remain a natural question and an open problem.  Every example of a non-reconstructible compact metrizable space that we present appears to have a \emph{unique} non-homeomorphic reconstruction.
\begin{myquest}
Is there an example of a compact space with more than one non-homeomorphic reconstruction? What is the maximal number of non-homeomorphic reconstructions of a compact metrizable space?
\end{myquest}
And second, our techniques used to establish the Reconstruction Characterisation Theorem rely on the spaces being metrizable.
\begin{myprob}
Find a characterisation of the reconstructible (first countable) compact spaces.  
\end{myprob}


\section{Cards with maximal finite compactifications}
\label{section3}

We prove that a compact metrizable space has a card with a maximal finite compactification if and only if the space does not contain a dense $G_\delta$ (a countable intersection of open sets) of 1-point components. It follows that non-reconstructible compact metrizable spaces are highly disconnected. 
 
We need the following two basic results from continuum theory. The component of a point $x \in X$ is denoted by $C_X(x)$, or $C(x)$ when $X$ is clear from context. 
\begin{mylem}[Second \v{S}ura-Bura Lemma, {\cite[A.10.1]{Mill01}}]
\label{adaptsurabura}
A component $C$ of a compact Hausdorff space $X$ has a clopen neighbourhood base in $X$ (for every open set $U \supseteq C$ there is a clopen set $V$ such that $C \subseteq V \subseteq U$). 
\end{mylem}

\begin{mylem}[Boundary Bumping Lemma {\cite[6.1.25]{Eng}}]
\label{boundarybumping}
The closure of every component of a non-empty proper open subset $U$ of a Hausdorff continuum intersects the boundary of $U$, i.e.\ $\closure{C_U(x)} \setminus U \neq \emptyset$ for all $x \in U$. 
\end{mylem}

A Hausdorff compactification $\gamma X$ of a space $X$ is called an $N$\emph{-point compactification} (for $N \in \N$) if its remainder $\gamma X \setminus X$ has cardinality $N$.  A \emph{finite compactification} of $X$ is an $N$-point compactification for some $N \in \N$. We say $\nu X$ is a \emph{maximal} $N$-point compactification if no other finite compactification $\gamma X$ has a strictly larger remainder, i.e.\ whenever $\cardinality{\gamma X \setminus X}=M$ then $M \leq N$.

For $N \geq 1$, a point $x \in X$ is $N$-\emph{splitting} in $X$ if $X \setminus \singleton{x} = X_1 \oplus \ldots \oplus X_N$ such that $x \in \closure{X_i}$ for all $i \leq N$. Further, we say that $x$ is \emph{locally $N$-splitting} (in $X$) if there exists a neighbourhood $U$ of $ x$, i.e.\ a set $U$ with $x \in \interior{U}$, such that $x$ is $N$-splitting in $U$. Moreover, a point $x$ is \emph{$N$-separating} in $X$ if $X \setminus \singleton{x}$ has a disconnection into $N$ (clopen) sets $ A_1 \oplus \cdots \oplus A_N$ such that all $A_i$ intersect $C_X(x)$. Similarly, we say $x$ is \emph{locally $N$-separating} in $X$ if there is a neighbourhood $U$ of $x$ such that $x$ is $N$-separating in $U$. 

We will also need the following three results from \cite{part1}.
\begin{mylem}[{\cite[3.4]{part1}}]
\label{lem:easytranslation}
A card $X \setminus \singleton{x}$ of a compact Hausdorff space $X$ has an $N$-point compactification if and only if $x$ is locally $N$-splitting in $X$.
\end{mylem}

\begin{mylem}[{\cite[4.8]{part1}}]
\label{fintranslation}
A card $X \setminus \singleton{x}$ of a Hausdorff continuum $X$ has an $N$-point compactification if and only if $x$ is locally $N$-separating in $X$.
\end{mylem}

\begin{mythm}[{\cite[4.11]{part1}}]
\label{bigreconstructionresult2}
The number of locally $3$-separating points in a $T_1$ space $X$ does not exceed the weight of $X$. 
\end{mythm}



Now we generalize the second of the above results beyond continua.
\begin{mylem}
\label{boundedaway}
Let $X$ be a compact metric space and suppose there exists $\delta >0$ such that the diameter of every component is at least $\delta$. Then $X \setminus \singleton{x}$ has an $N$-point compactification if and only if $x$ is locally $N$-separating in $X$.
\end{mylem}


\begin{proof}
The backwards direction is immediate from Lemma~\ref{lem:easytranslation}, so we focus on the direct implication. Let us assume that $X\setminus \singleton{x}$ has an N-point compactification. By Lemma~\ref{lem:easytranslation}, there is a compact neighbourhood $U=\closure{B_\epsilon(x)}$ of $x$ such that $U \setminus \singleton{x} = U_1 \oplus \cdots \oplus U_N$ and $x \in \closureIn{U_i}{X}$ for all $i$. Assume without loss of generality that $\epsilon < \delta/2$. We have to show that $C_{U}(x)$, the component of $x$ in $U$, intersects every $U_i$. Suppose for a contradiction this is not the case, i.e.\ that say $C(x) \cap U_1 = \emptyset$. Since $x$ lies in the closure of $U_1$, we can find a sequence $\set{x_n}:{n\in\w} \subset U_1$ converging to $x$.

Note that if $x \in \closure{C_{U_1}(x_n)}$ then $C_{U_1}(x_n) \subset C_U(x)$, a contradiction. Next, since $\diam{C_X(x_n)} > 2 \epsilon$, we have $C_X(x_n) \not\subseteq U$. Hence, applying Lemma~\ref{boundarybumping} to $C_X(x_n) \cap U_1$ yields that $C_{U_1}(x_n)$ limits onto the boundary of $U$. Hence, whenever $d(x,x_n) < \epsilon/2$ then $\textnormal{diam}(C_{U_1}(x_n)) \geq\epsilon/2$.

To conclude the proof, note that $C_U(x) \cap U_1 = \emptyset$ implies that $x$ is a one-point component of $\closure{U_1}$. By Lemma~\ref{adaptsurabura} there is a clopen neighbourhood $V$ of $x$ in $\closure{U_1}$ such that the diameter of $V$ is at most $\epsilon/4$. However, we have $x_n \in V$ for $n$ large enough, implying that $\textnormal{diam}(C_{U_1}(x_n)) \leq \textnormal{diam}(V) \leq \epsilon/4$, a contradiction.
\end{proof}

Our characterisation of compact metrizable spaces having a card with a maximal finite compactification follows.

\begin{mythm}
\label{charfinitecompactificationcptmetr}
For a compact metrizable space $X$ without isolated points the following are equivalent:
\begin{enumerate}
\itemsep0em 
\item $X$ has a card with a maximal $1$- or $2$-point compactification,  
\item $X$ has a card with a maximal finite compactification,
\item $C_1(X)$, the set of $1$-point components of $X$, is not dense in $X$, and
\item $C_1(X)$ does not form a dense $G_\delta$ of size $\cont$.
\end{enumerate}
\end{mythm}

\begin{proof}
The implications $(1) \Rightarrow (2)$ and $(3) \Rightarrow (4)$ are trivial. 

For $(2) \Rightarrow (3)$, assume $C_1(X)$ is dense in $X$. Let $x$ be an arbitrary point of $X$. We need to show that the card $X \setminus \singleton{x}$ has arbitrarily large finite compactifications.

Fix a nested neighbourhood base $\set{V_n}:{n \in \N}$ of $x$ such that $\closure{V_{n+1}} \subsetneq V_n$. Since the 1-point components are dense, it follows from Lemma~\ref{adaptsurabura} that for all $n$ there is a non-empty clopen set $F_n \subsetneq V_n \setminus \closure{V_{n+1}}$. It follows that whenever $A_1, \ldots, A_{N-1}$ partitions $\N$ into infinite subsets, the sets
$$ G_i = \bigcup_{n \in A_i} F_n \quad \text{and} \quad G_0 = (X \setminus \singleton{x}) \setminus  \bigcup_{n \in \N} F_n$$ 
form a partition of $X \setminus \singleton{x}$ into $N$ non-empty non-compact clopen subsets. By Lemma~\ref{lem:easytranslation}, the card $X \setminus \singleton{x}$ has an $N$-point compactification. Since $N$ was arbitrary, the card $X \setminus \singleton{x}$ does not have a maximal finite compactification.

For $(4) \Rightarrow (1)$, we will prove the contrapositive. Assume that all cards of $X$ have a three-point compactification. Set, for $n \in \N$, $X_n = \set{x \in X}:{\diam{C(x)} \geq 1/n}$.

Every $X_n$ is a closed subset of $X$. To see this, suppose for a contradiction that $x$ lies in the closure of $X_n$ and $\diam{C(x)} < 1/n$. Find $\epsilon > 0$ such that $\diam{C(x)} + 2\epsilon < 1/n$. By the Lemma~\ref{adaptsurabura}, there is a clopen set $F$ of $X$ such that  
$C(x) \subseteq F \subseteq B_\epsilon(C(x))$. 
Since $F$ is a neighbourhood of $x$, there is a point $y \in F \cap X_n$ witnessing that $x$ lies in the closure of $X_n$. It follows $y \in C(y) \subseteq F$, and hence $\diam{C(y)}< 1/n$, contradicting $y \in X_n$. This shows that $X_n$ is closed.

We now argue that every $X_n$ has empty interior in $X$. Otherwise, $\interior{X_n}$ would have to be uncountable, being locally compact without isolated points.  As $X \setminus \singleton{x}$ has a $3$-point compactification for all $x$, Lemma~\ref{lem:easytranslation} implies that $X_n \setminus \singleton{x}$ has a $3$-point compactification for all $x \in \interior{X_n}$. However, since all components of $X_n$ have diameter at least $1/n$, Lemma~\ref{boundedaway} implies that all $x \in \interior{X_n}$ are locally $3$-separating in $X_n$. But as compact metrizable spaces have countable weight, no such space can contain uncountably many locally $3$-separating points by Lemma~\ref{bigreconstructionresult2}, a contradiction.  Thus, every $X_n$ has empty interior in $X$.

Finally, $C_1(X)=\bigcap_{n \in \N} X\setminus X_n$. Since every $X\setminus X_n$ is open dense, the Baire Category Theorem  implies that $C_1(X)$ is a dense $G_\delta$ in $X$. In particular, $C_1(X)$ is completely metrizable  without isolated points, and thus has cardinality $\cont$. 
\end{proof}

Since every compact Hausdorff space with a card with a maximal finite compactification is reconstructible, \cite[3.11]{part1}, we deduce:
\begin{mythm}
\label{lotsofreconstrutible}
Every compact metrizable space in which the union of all $1$-point components of $X$ does not form a dense $G_\delta$ of cardinality $\cont$ is reconstructible. \qed
\end{mythm}

\begin{mycor}
\label{Cantorchar}
The Cantor set is characterised topologically as the unique compact metrizable homogeneous non-reconstructible space. \qed
\end{mycor}

\section{Universal sequences of clopen sets}
\label{section4}

 In this section, we formally introduce universal sequences, show that every non-reconstructible compact metrizable space has a universal sequence, and provide a sufficient condition for the converse.

Recall that a sequence  $\Sequence{B_n}:{n \in \N}$ of subsets of a space $X$ is said to converge to a point $x$ ($B_n \to x$) if for every neighbourhood $U$ of $x$ there exists $N \in \N$ such that $B_n \subseteq U$ for all $n \geq N$. 
Suppose $\sequence{T}$ and $\Sequence{T_n}:{n \in \N}$ are sequences of topological spaces. A sequence $\Sequence{B_n}:{n \in \N}$ is of \emph{type} $\sequence{T_n}$ if $B_n \cong T_n$ for all $n\in \N$, and is of \emph{constant type} $\sequence{T}$ if $B_n \cong T$ for all $n \in \N$. 

We say that a topological space $X$ has a \emph{universal sequence of type} $\sequence{T_n}$ if every point of $X$ is the limit of a sequence of disjoint clopen sets of type $\sequence{T_n}$. The abbreviation `$X$ has a universal sequence' means that for some type $\Sequence{T_n}:{n\in\N}$, $X$ has a universal sequence of that type $\sequence{T_n}$. Lastly, $X$ has a \emph{constant universal sequence} if there is $\sequence{T}$ such that every point of $X$ is the limit of a sequence of disjoint clopen sets of constant type $\sequence{T}$.

\begin{myprop}[Universal Sequence Existence]
\label{universalsequence}
Every non-reconstructible compact metrizable space has a universal sequence.
\end{myprop}

\begin{proof}
Let $X$ be a compact metrizable space with a non-homeomorphic reconstruction $Z$. Then $Z$ is non-compact  \cite[5.2]{recpaper}. It follows from results of \cite{recpaper}, but in this case can also be proved directly, that $Z$ is  locally compact, separable and metrizable. So the $1$-point compactification, $\alpha Z$, of $Z$,  is metrizable. Moreover, since $X$  contains $1$-point components by  Theorem~\ref{lotsofreconstrutible}, every component of $Z$ is compact. In particular, the point $\infty \in \alpha Z$ is a $1$-point component. By Lemma~\ref{adaptsurabura}, the point $\infty$ has a (countable) neighbourhood base of clopen sets in $\alpha Z$, and hence $Z$ can be written as $\bigoplus_{n \in \N} T_n$ for disjoint clopen compact subsets $T_n \subset Z$.  We will show that a tail of $\Sequence{T_n}:{n \in \N}$ is the type of the desired universal sequence.

\begin{myclmn}
For every $x \in X$ there is $N_x \in \N$ such that $x$ is the limit of a sequence of disjoint clopen sets of type $\Sequence{T_n}:{n > N_x}$. 
\end{myclmn}

To see the claim, note that $X \setminus \singleton{x} \cong Z \setminus \singleton{z}$ for some suitable $z \in Z$. But if $z \in T_{N_x}$ then $\Sequence{T_n}:{n > N_x}$ is an infinite sequence of disjoint compact clopen sets in $X \setminus \singleton{x}$ such that their union is a closed set. Compactness of $X$ now implies that $\Sequence{T_n}:{n > N_x}$ converges to $x$.

\begin{myclmn}
There is $N \in \N$ such that every $x \in X$ is the limit of a sequence of disjoint clopen sets of type $\Sequence{T_n}:{n > N}$. 
\end{myclmn}

Choose non-empty open sets $U$ and $V$ with disjoint closures and $N \in \N$ such that both $U$ and $V$ contain a copy of $\bigoplus_{n > N} T_n$. If $x \in X \setminus \closure{U}$, apply the previous claim and use the copies of $T_n$ for $N < n \le N_x$ in $U$ to obtain a sequence as claimed. By symmetry, the same holds for $x \in X \setminus \closure{V}$, completing the proof.  
%
\end{proof}

If a space $X$ has a universal sequence of type $\sequence{T_n}$, there is a witnessing indexed family $\set{\script{U}_x}:{x \in X}$ where every $\script{U}_x = \Sequence{U^x_n}:{n \in \N}$ is a sequence of disjoint clopen subsets of $X$ of type $\sequence{T_n}$ converging to $x$. Let us call such an indexed family a \emph{universal sequence system (of type $\sequence{T_n}$)}. We distinguish the following additional properties. A universal sequence system $\set{\script{U}_x}:{x \in X}$
\begin{itemize}
\item is \emph{complement-equivalent} if for all $x,y \in X$, and $N \in \N$ with $\bigcup_{n>N} U^x_n \cap  \bigcup_{n>N} U^y_n=\emptyset$, we have $X \setminus  \bigcup_{n>N} U_n^x \cong X \setminus  \bigcup_{n>N} U_n^y$,
\item has \emph{augmentation} if detaching a sequence $\script{U}_x$ from its unique limit point $x$ and making it converge onto any $y \in X$ produces a space homeomorphic to $X$, i.e.\ for all $x,y \in X$, 
we have 
$$\p{\p{X \setminus \bigcup\script{U}_x} \oplus  \alpha\p{\bigcup \script{U}_x}}/\Set{y,\infty} \cong X,$$
\item is \emph{point-fixing} if for all $x,y \in X$ and $N \in \N$ with $\bigcup_{n>N} U^x_n \cap  \bigcup_{n>N} U^y_n=\emptyset$, there is a homeomorphism $f_{xy} \colon X \setminus \bigcup_{n>N} U_n^x \to X \setminus \bigcup_{n>N} U_n^y$ fixing the points $x$ and $y$, and
\item is \emph{thin} if for all $x \in X$, the set $\singleton{x} \cup \bigcup \mathcal{U}_x$ is not a neighbourhood of $x$. 
\end{itemize}

\begin{mylem}
\label{pointfixing}
Every point-fixing universal sequence system is complement-equivalent and has augmentation.
\end{mylem}
\begin{proof}
Complement-equivalent is clear. Let $x,y \in X$. As $\p{X \setminus \bigcup \script{U}_x} \oplus \bigcup \script{U}_x = \p{X \setminus \bigcup_{n>N} U_n^x} \oplus \bigcup_{n>N} U_n^x$ for all $N \in \N$, we may assume $y \notin \bigcup \script{U}_x$. Then
\begin{align*}
X & \cong \p{\p{X \setminus  \bigcup \script{U}_{y}} \oplus \alpha \p{\bigcup \script{U}_{y}}}/\Set{y,\infty} \\
&\cong \p{\p{X \setminus  \bigcup \script{U}_{y}} \oplus \alpha \p{\bigcup \script{U}_{y}}}/\Set{f_{xy}(y),\infty} \quad \textnormal{(as $f_{xy}$ fixes $y$)}\\ 
&\cong \p{\p{X \setminus  \bigcup \script{U}_{x}} \oplus \alpha \p{\bigcup \script{U}_{y}}}/\Set{y,\infty} \quad \textnormal{(as $f_{xy}$ is a homeomorphism)} \\ 
&\cong \p{\p{X \setminus  \bigcup \script{U}_{x}} \oplus \alpha \p{\bigcup \script{U}_{x}}}/\Set{y,\infty} \quad \textnormal{(as $\bigcup \script{U}_{y} \cong \bigcup \script{U}_x$).} \qedhere 
\end{align*}
\end{proof}

Recall that a space is \emph{pseudocompact} if every discrete family of open sets is finite. For Tychonoff spaces this coincides with the usual definition that every continuous real-valued function is bounded \cite[\S 3.10]{Eng}. Evidently, compact spaces are pseudocompact. 

\begin{mylem}
\label{thm:definitiveNonRec}
Every pseudocompact Hausdorff space with a complement-equivalent universal sequence system with augmentation (in particular: with a point-fixing universal sequence system) is non-reconstructible. 
\end{mylem}

\begin{proof}
Let $X$ be a pseudocompact Hausdroff space with a universal sequence system $\set{\script{U}_x}:{x \in X}$ which is complement-equivalent with augmentation. We verify that $Z=\p{X \setminus \bigcup \script{U}_x} \oplus \bigcup \script{U}_x$, or equivalently $\p{X \setminus \bigcup_{n>N} U_n^x} \oplus \bigcup_{n>N} U_n^x$, is a non-pseudocompact---and hence non-homeomorphic---reconstruction of $X$. Indeed, it is non-pseudocompact, as $\script{U}_x$ is an infinite discrete family of open sets in $Z$.

The inclusion ``$ \singletonDeletion{X} \subseteq \singletonDeletion{Z}$" follows from complement-equivalence. Consider a card $X \setminus \singleton{y}$. We may assume $\bigcup \script{U}_x \cap \bigcup \script{U}_y = \emptyset$. Then for some suitable $z \in X$,
\[X \setminus \singleton{y} =  X \setminus \p{ \singleton{y} \cup \bigcup \script{U}_y} \oplus \bigcup \script{U}_y \cong X \setminus \p{ \singleton{z} \cup \bigcup \script{U}_x} \oplus \bigcup \script{U}_x.\]

The inclusion ``$ \singletonDeletion{X} \supseteq \singletonDeletion{Z}$" follows from augmentation. Consider a card $Z \setminus \singleton{y}$. Since $\p{\alpha Z}/\Set{\infty,y} \cong X$, we have for some suitable $z \in X$ that
\[Z \setminus \singleton{y} =  \p{\p{\alpha Z}/\Set{\infty,y}} \setminus \singleton{\Set{\infty,y}} \cong X \setminus \singleton{z}. \qedhere\]
 \end{proof}

\section{Universal sequences imply non-reconstructibility}

In this section we prove the forward implication of the Reconstruction Characterisation Theorem~\ref{recchartthm}.

\begin{myprop}
\label{recchartlemma}
A compact metrizable space with a universal sequence is non-reconstructible.
\end{myprop}

In order to apply Lemma~\ref{thm:definitiveNonRec}, we will show that every compact metrizable space with a universal sequence has a universal sequence \emph{refinement} with a corresponding point-fixing universal sequence system. Indeed we show in Lemma~\ref{thinuniversalsequencesystem} that every compact metrizable space with a universal sequence has a thin universal sequence system. Lemmas~\ref{subuniversalsequences}, \ref{Rolfslemma} and~\ref{constructpointfixhomeo}, imply that this thin system can be refined to a (thin) point-fixing universal sequence system, as required. 

Here, a sequence of topological spaces $\Sequence{T'_n}:{n \in \N}$ is a \emph{refinement} of $\Sequence{T_n}:{n \in \N}$ if there is an injective map $\phi \colon \N \to \N$ such that $T'_n \hookrightarrow T_{\phi(n)}$ embeds as a clopen subset for all $n \in \N$. Note that if $\set{\script{U}_x}:{x \in X}$ is a universal sequence system of type $\sequence{T_n}$, then any refinement $\sequence{T'_n}$ of $\sequence{T_n}$ naturally induces a refinement $\set{\script{U}'_x}:{x \in X}$ of $\set{\script{U}_x}:{x \in X}$.

%

\begin{mylem}
\label{subuniversalsequences}
Any refinement of a (thin) universal sequence system is again a (thin) universal sequence system. \qed
\end{mylem}

\begin{mylem}
\label{thinuniversalsequencesystem}
Every universal sequence system has a thin universal sequence system refinement. 
\end{mylem}
\begin{proof}
For a universal sequence system of type $\sequence{T_n}$, the refinement $\Sequence{T_{2n}}:{n \in \N}$ is easily seen to induce a thin universal sequence system.
\end{proof}


In the following two lemmas, we use the shorthand ``$x$ has a neighbourhood basis homeomorphic to $\bigoplus V_n \to x$" to mean: for every neighbourhood $U$ of $x$ there is $N \in \N$ and a clopen $V$ with $x \in V \subseteq U$ such that $V\setminus \singleton{x} \cong \bigoplus_{n > N} V_n$.

\begin{mylem}
\label{Rolfslemma}
Suppose $X$ is compact metrizable with a universal sequence of type $\sequence{T_n}$. Then $X$ has a universal sequence refinement $\sequence{T'_n}$ such that a dense collection of  points $x$ have neighbourhood basis homeomorphic to $\bigoplus T_n' \to x$.
\end{mylem}
\begin{proof}
Every non-empty open set contains a clopen set from the tail of the universal sequence. Thus, by recursion, we can find a nested collection of non-empty clopen sets $F_k$ of vanishing diameter such that $F_k \cong T_{n_k}$ for a subsequence $\Sequence{n_k}:{k \in \N}$. 

Compactness and $\diam{F_k} \to 0$ imply that $\Intersection F_k = \singleton{x}=C(x)$ for some point $x \in X$. By the \v{S}ura-Bura Lemma, $x$ has a neighbourhood basis homeomorphic to $\bigoplus F_k \setminus F_{k+1} \to x$. Moreover, since every copy of $T_{n_k}$ contains a point with this property, there is a dense set of points $y$ with a neighbourhood looking like $\bigoplus F_k \setminus F_{k+1} \to y$.

Thus, the refinement $\Sequence{T'_k=T_{n_k} \setminus T_{n_{k+1}}}:{k \in \N}$ of $\sequence{T_n}$ is as required.
\end{proof}

Comparing the above lemma with Theorem \ref{lotsofreconstrutible}, we note that the points we construct are indeed one-point components and hence having a universal sequence implies that the set of one-point components is dense in $X$.

\begin{mylem}
\label{constructpointfixhomeo}
Every thin universal sequence system $\set{\script{U}_x}:{x \in X}$ of type $\sequence{T_n}$ in a compact metrizable space $X$, with the property that a dense collection of points $z$ has neighbourhoods looking like $\bigoplus T_n \to z$, is point-fixing.
\end{mylem}

\begin{proof}
Let $x,y \in X$ arbitrary. Without loss of generality we may assume $\bigcup \script{U}_x \cap \bigcup \script{U}_y = \emptyset$. To see that $\set{\script{U}_x}:{x \in X}$ is point-fixing, we need to construct a homeomorphism $f\colon X \setminus \bigcup \script{U}_x \to X \setminus \bigcup \script{U}_y$ with $f(x)=x$ and $f(y)=y$.

By thinness, the space $Z_0 = X \setminus \p{\bigcup \script{U}_x \cup \bigcup \script{U}_y \cup \Set{x,y}}$ is a non-empty open subset of $X$, so by assumption there is $z_0$ and a clopen $V_{0}=\singleton{z_0} \cup \bigoplus_{n > n_0} V^{z_0}_n$ with $V^{z_0}_n \cong T_n$ for all $n>n_0$ such that $z_0 \in V_0 \subset Z_0$. Put $I_0=[0,n_0] \cap \N$ and consider a homeomorphism $$g_0 \colon \bigcup_{n \in I_0} U^y_n \mapsto \bigcup_{n \in I_0} U^x_n.$$ 
Next, consider the open subspace $Z_1 = X \setminus \p{\bigcup \script{U}_x \cup \bigcup \script{U}_y \cup V_{0} \cup \Set{x,y}}$. Again by thinness, there are $z_{-1}$ and $z_1$ with disjoint clopen neighbourhoods $V_{-1}=\singleton{z_{-1}} \cup \bigoplus_{n > n_1} V^{z_{-1}}_n$ and $V_{1}=\singleton{z_{1}} \cup \bigoplus_{n > n_1} V^{z_1}_n$ with $n_1 > n_0$ and  $V^{z_{-1}}_n \cong T_n \cong V^{z_{1}}_n$ for $n>n_1$ such that $V_{-1} \subset B_{\frac12}(y) \cap Z_1$ and $V_{1} \subset B_{\frac12}(x) \cap Z_1$. 
Put $I_1=(n_0,n_1] \cap \N$ and define 
$$g_1 \colon \bigcup_{n \in I_1} U^y_n \mapsto  \bigcup_{n \in I_1}  V^{z_0}_n \mapsto  \bigcup_{n \in I_1} U^x_n.$$ 
Continuing recursively, we get a double sequence $\set{z_k}:{k \in \Z}$ with disjoint clopen neighbourhoods $\set{V_k}:{k \in \Z}$ such that $ V_k \xrightarrow[k \to -\infty]{} y$ and $V_k \xrightarrow[k \to \infty]{} x$, and maps $g_{k}$ for intervals $I_{k}=(n_{k-1}, n_{k}] \cap \N$ sending
$$g_{k} \colon \bigcup_{n \in I_{k}} U^y_n \mapsto  \bigcup_{n \in I_{k}} V^{z_{1-k}}_n \mapsto \bigcup_{n \in I_{k}}  V^{z_{2-k}}_n \mapsto \cdots \mapsto \bigcup_{n \in I_{k}}  V^{z_{k-1}}_n \mapsto   \bigcup_{n \in I_{k}} U^x_n.$$

\begin{center}
\begin{figure}[ht]
\begin{tikzpicture}[>=stealth,scale=0.9]

\draw (-4,0) -- (-4.5,4) -- (-3.5,4) -- cycle;
\draw (-4.125,1) -- (-3.875,1);
\draw (-4.25,2) -- (-3.75,2);
\draw (-4.375,3) -- (-3.625,3);

\draw (4,0) -- (4.5,4) -- (3.5,4) -- cycle;
\draw (4.125,1) -- (3.875,1);
\draw (4.25,2) -- (3.75,2);
\draw (4.375,3) -- (3.625,3);

\draw (0,0) -- (0.375,3) -- (-0.375,3) -- cycle;
\draw (-0.125,1) -- (0.125,1);
\draw (-0.25,2) -- (0.25,2);

\draw (-1.5,0) -- (-1.75,2) -- (-1.25,2) -- cycle;
\draw (-1.625,1) -- (-1.375,1);

\draw (1.5,0) -- (1.75,2) -- (1.25,2) -- cycle;
\draw (1.625,1) -- (1.375,1);

\draw (-2.5,0) -- (-2.625,1) -- (-2.375,1) -- cycle;

\draw (2.5,0) -- (2.625,1) -- (2.375,1) -- cycle;

\draw[dotted] (-4,0) -- (-3.25,0);

\draw[dotted] (4,0) -- (3.25,0);

\draw[fill=red] (-4,0) circle (0.1);
\draw (-4,0) node[anchor=north]  {$y$};

\draw[fill=blue] (4,0) circle (0.1);
\draw (4,0) node[anchor=north]  {$x$};

\draw[fill=green] (0,0) circle (0.1);
\node[anchor=north] (0,0) {$z_0$};

\draw[fill=green] (-1.5,0) circle (0.1);
\draw (-1.5,0) node[anchor=north]  {$z_{-1}$};

\draw[fill=green] (-2.5,0) circle (0.1);
\draw (-2.5,0) node[anchor=north]  {$z_{-2}$};

\draw[fill=green] (-3.25,0) circle (0.1);
\draw (-3.25,0) node[anchor=north]  {$z_{-3}$};

\draw[fill=green] (1.5,0) circle (0.1);
\draw (1.5,0) node[anchor=north]  {$z_{1}$};

\draw[fill=green] (2.5,0) circle (0.1);
\draw (2.5,0) node[anchor=north]  {$z_{2}$};

\draw[fill=green] (3.25,0) circle (0.1);
\draw (3.25,0) node[anchor=north]  {$z_{3}$};

\draw[->] (-3.8,3.5) .. controls (0,4.25) .. (3.8,3.5);

\draw[->] (-3.9,2.5) .. controls (-2,2.75) .. (-0.1,2.5);

\draw[->] (0.1,2.5) .. controls (2,2.75) .. (3.9,2.5);

\draw[->] (-3.95,1.5) .. controls (-3,1.75) .. (-1.5,1.5);
\draw[->] (-1.45,1.5) .. controls (-0.75,1.75) .. (0,1.5);
\draw[->] (0.1,1.5) .. controls (0.75,1.75) .. (1.45,1.5);
\draw[->] (1.5,1.5) .. controls (3,1.75) .. (3.95,1.5);

\end{tikzpicture}
\setlength{\belowcaptionskip}{-10pt}
\caption{Constructing the maps $g_k$}
\end{figure}
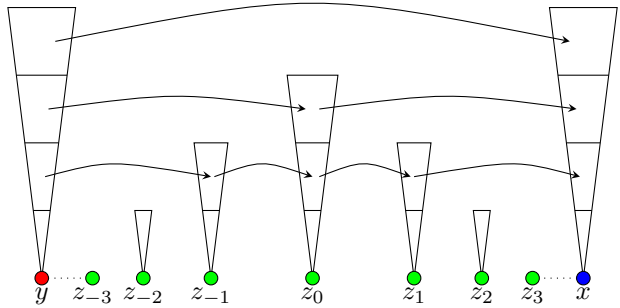
\end{center}

Once the recursion is completed, consider the closed sets
$$ A = \bigcup \script{U}_y \cup \bigcup_{n \in \Z} V_{n} \cup \Set{x,y} \quad \textnormal{and} \quad B = X \setminus \p{\bigcup \script{U}_x \cup \bigcup \script{U}_y \cup \bigcup_{n \in \Z} V_{n}}.$$
The resulting map $g=\bigcup g_k$ is continuous on $A$, and sends $g(y)=y$, $g(z_k)=z_{k+1}$ and $g(x)=x$. In particular, the partial maps $g |_A$ and $\operatorname{id} |_B$ agree on $A \cap B = \Set{x,y}$. Thus the map 
$f=g |_A \cup \operatorname{id} |_B \colon X \setminus \bigcup \script{U}_x \to X \setminus \bigcup \script{U}_y$ is a continuous bijection, and hence, by compactness, a homeomorphism.
\end{proof}

\section{Further reconstruction results}
\label{section5}

%
%
%
%

In this section we apply the Reconstruction Characterisation Theorem~\ref{recchartthm} to better understand the role of the set $C_1(X)$ of $1$-point components in the reconstruction problem. Theorem~\ref{lotsofreconstrutible} showed that if $X$ is compact, metrizable and non-reconstructible then $C_1(X)$ is a dense $G_\delta$ subset of $X$. Now we investigate conditions on $X$ or on $C_1(X)$ ensuring that the converse holds. 

The characterisation of non-reconstructibility for compact metrizable spaces seems to depend crucially on metrizability. The results of this section also yield interesting sufficient conditions for non-reconstructibility in general compact Hausdorff spaces.

Recall that a  space $X$ is  \emph{h-homogeneous}, or  \emph{strongly homogeneous}, if every non-empty clopen subset of $X$ is homeomorphic to $X$. 
\begin{mythm}
\label{stronghomogeneousnonrec}
Every h-homogeneous, first-countable, compact space  in which the $1$-point components are dense is non-reconstructible.
\end{mythm}

 A collection $\script{B}$ of open sets is called a $\pi$-base for $X$ if for every open set $U \subset X$ there is $B \in \script{B}$ such that $B \subset U$. We call a space $\pi$\emph{-homogeneous} if it has a clopen $\pi$-base of pairwise homeomorphic elements.  If a space $X$ is $h$-homogeneous and has a dense set of $1$-point components then  the Second \v{S}ura-Bura Lemma~\ref{adaptsurabura} implies that $X$ is $\pi$-homogeneous. Also note that every non-trivial $h$-homogenous space has no isolated points. The following observation relating $\pi$-homogeneity and universal sequences in the realm of first-countable spaces is straightforward.
\begin{mylem}
\label{pihomsequences}
Let $X$ be a first-countable Hausdorff space without isolated points. Then $X$ is $\pi$-homogeneous if and only if $X$ has a constant universal sequence. \qed
\end{mylem}
To prove Theorem~\ref{stronghomogeneousnonrec} it remains to verify the following result which is of interest in its own right.
\begin{mythm}
\label{supernonreconstruction}
Every pseudocompact Hausdorff space with a constant universal sequence is non-reconstructible.
\end{mythm}

\begin{proof}
If $X$ has a universal sequence system $\set{\script{U}_x}:{x \in X}$ of constant type then $\script{U}'_x=\Sequence{U^x_{2n}}:{n \in \N}$ gives a point-fixing universal sequence system. Indeed, the map $ X \setminus \p{\bigcup \script{U}'_x} \to  X \setminus \p{\bigcup \script{U}'_y} $ sending $U^x_{2n+1}$ homeomorphically to $U^x_{n}$ and $U^y_{n}$ homeomorphically to $U^y_{2n+1}$, and being the identity elsewhere (in particular: $x \mapsto x$, $y \mapsto y$) is a homeomorphism. So the claim follows from Lemma~\ref{thm:definitiveNonRec}.
\end{proof}

%

An example of a pseudocompact non-compact space with a constant universal sequence is mentioned in \cite[p.19]{matveev}. Start with the usual (pseudocompact) $\Psi$-space on $\w$ \cite[3.6.I]{Eng}, and replace every isolated point by a clopen copy of the Cantor set. Every $\Psi$-space is first-countable, so this gives a pseudocompact space with a constant universal sequence of type $\sequence{C}$. 

Now we consider the case when $C_1(X)$ is not just a $G_\delta$ set, but is open or contains a non-empty open subset (in other words, is dense but not co-dense). 
\begin{mythm}
\label{tricktrick}
For a compact metrizable space $X$ in which $C_1(X)$ forms a dense $G_\delta$ with non-empty interior the following are equivalent:
\begin{enumerate}
\item $X$ is not reconstructible,

\item $X$ has a constant universal sequence (of type $\sequence{C}$),

\item the interior of $C_1(X)$ contains no isolated points and is dense in $X$, and

\item $X$ is a compactification of $C \setminus \{0\}$.
\end{enumerate}
\end{mythm}

\begin{proof} First we show the equivalence of (1), (2) and (3). 
If the interior of $C_1(X)$ is non-empty and not dense in $X$, there cannot exist a universal sequence: any sequence of disjoint clopen sets converging to a point in the interior of $C_1(X)$ is eventually of type $\sequence{C}$, which is not the case for a point outside the closure of $\interior{C_1(X)}$. Thus, $X$ is reconstructible by the Reconstruction Characterisation Theorem, so (1) fails. Equivalently, (1) implies (3).

If (3) holds, the clopen sets of $X$ homeomorphic to the Cantor set form a $\pi$-base for $X$. Thus, $X$ has a constant universal sequence of type $\sequence{C}$ by Lemma~\ref{pihomsequences} and (2) holds. And (2) implies (1)  by Theorem~\ref{supernonreconstruction}.

Now we verify (3) implies (4). Let $U$ be an open dense, proper subset of $X$ contained in $C_1(X)$. Then $U$ is a zero-dimensional, locally compact, non-compact metrizable space without isolated points, and hence homeomorphic to $C\setminus\{0\}$ \cite[6.2.A(c)]{Eng}, and $X$ is a compactification of this latter set.

It remains to show (4) implies (3). But if $X$ contains a dense subset $A$ homeomorphic to $C\setminus \{0\}$, then $A$ is open by local compactness, dense, without isolated points, and contained in $C_1(X)$.
\end{proof}

\section{Building non-reconstructible compact spaces}
\label{section6}

This final section contains examples of compact metrizable spaces illustrating the frontier between reconstructibility and non-reconstructibility. One objective is to present a broad variety of non-reconstructible compact metrizable spaces. A second objective is to show that some natural strengthenings of the Reconstruction Characterisation do not hold.

Consider a compact metrizable space $X$ without isolated points. 
We know that if $X$ is not reconstructible then $C_1(X)$, the set of $1$-point components in $X$, is a dense $G_\delta$. If $X$ is h-homogeneous then density of $C_1(X)$ suffices for non-reconstructibility (via a constant universal sequence). And non-reconstructibility (again via constant universal sequences) also follows if $C_1(X)$ is a dense \emph{open} $G_\delta$. 
Certain questions now arise:
\begin{enumerate}
\item Is $X$  non-reconstructible if (and only if) $C_1(X)$ is a dense $G_\delta$?  
\item If $X$ is non-reconstructible then is $C_1(X)$ a dense open (or dense, not co-dense) set? Equivalently, is $X$ non-reconstructible if and only if it is the compactification of $C\setminus \{0\}$?
\item If $X$ is non-reconstructible, must it have a \emph{constant} universal sequence?
\end{enumerate}
We give negative answers to all these below, providing h-homogeneous examples where possible. 

\subsection{Non-reconstructible spaces with open dense $C_1(X)$: Compactifications of $C \setminus \singleton{0}$}\label{section61}  

There are many non-reconstructible spaces $X$ where $C_1(X)$ is dense but not co-dense.  
\begin{mylem}
\label{cantorminuspointrec}
For every compact metrizable space $K$ there is a compact metrizable space $X=X_K$ with an open dense set $U$ homeomorphic to $C \setminus \{0\}$ such that $X \setminus U$ is homeomorphic to $K$. 
\end{mylem} 
Since all these $X_K$ are non-reconstructible, the variety of non-reconstructible compact metrizable spaces is the same as that of all compact metrizable spaces. Note that if $K$ is not zero-dimensional then $X_K$ is not $h$-homogeneous.

\begin{proof}
Note that $C\setminus \{0\}$ is homeomorphic to $\w \times C$ and  $\p{\w +1} \times C \cong C$, and fix a continuous surjection $f \colon \singleton{\w} \times C \to K$. The adjunction space $C \cup_f K$---the quotient $X_K=C / \script{P}$ for $\script{P}= \set{f^{-1}(x)}:{x \in K} \cup \set{\singleton{x}}:{x \in \w \times C}$---is a metrizable compactification of $\w \times C$ with remainder homeomorphic to $K$ \cite[A.11.4]{Mill01}.
\end{proof}

With an eye on general compact Hausdorff spaces we state and prove a more general construction.
Recall that a subset $D$ of $X$ is \emph{sequentially dense} if every point of $X$ is the limit of a converging sequence of points in $D$. A space is \emph{sequentially separable} if it has a countable, sequentially dense subset. In general, the \emph{sequential density} of a space is the least cardinal $\kappa$ such that there exists a subset of $X$ of size $\kappa$ which is sequentially dense in $X$.

In \cite{Tkachuk}, V. Tkachuk proved that every compact Hausdorff space $K$ is a remainder of a compactification of a discrete space of cardinality at most $\cardinality{K}$. The following result requires only minor modifications in Tkachuk's original proof.

\begin{mythm}[cf.\ Tkachuk \cite{Tkachuk}]
\label{Tkaschukstheorem}
For every compact Hausdorff space $K$ of sequential density $\kappa$ there is a compactification $X_K$ of the discrete space of size $\kappa$ such that 
\begin{enumerate}
\itemsep0em 
\item the remainder $X_K \setminus \kappa$ is homeomorphic to $K$, and
\item the discrete space $\kappa$ is sequentially dense in $X_K$. 
\end{enumerate}
\end{mythm}

%

\begin{mycor}
\label{CorTkaschukstheorem}
For every pair of compact Hausdorff spaces $K$ of sequential density $\kappa$, and $Z$ with a constant universal sequence of type $\sequence{T}$, there is a compactification $X_{K,Z}$ of $Z\times \kappa$ such that 
\begin{enumerate}
\itemsep0em 
\item the remainder of that compactification is homeomorphic to $K$, and
\item the compactification has a constant universal sequence of type $\sequence{T}$.
\end{enumerate}
In particular, this compactification is non-reconstructible.
\end{mycor}
\begin{proof}
In Theorem~\ref{Tkaschukstheorem}, replace every point of $\kappa$ by a clopen copy of $Z$. By sequential density, the resulting space has a constant universal sequence of type $\sequence{T}$. 
\end{proof}


\subsection{Non-reconstructible spaces with dense, co-dense $C_1$: $h$-homogeneous spaces}

We present a simple method to obtain non-reconstructible h-homogeneous spaces with dense and co-dense $C_1(X)$. This shows that our second question above has a negative answer. Indeed, note that in a $h$-homogeneous compact metrizable space different from the Cantor set, $C_1(X)$ must always be co-dense. 

In Lemma~\ref{C1dense_butrecon}, these spaces are modified to give an example of a compact metrizable space without isolated points which is reconstructible even though $C_1(X)$ is a dense and co-dense $G_\delta$---thus answering our first question in the negative.

Finding $h$-homogeneous spaces is simplified by the next result.
\begin{mythm}[Medini {\cite[Thm.\ 18]{medini}}]
\label{medini}
If a Hausdorff space $X$ has a dense set of isolated points then $X^\kappa$ is $h$-homogeneous for every infinite cardinal $\kappa$.
\end{mythm}

\begin{mycor}\label{firstctblehhom}
If $X$ is a first-countable compact space with a dense set of isolated points then $X^\N$ is h-homogeneous and non-reconstructible.
\end{mycor}

\begin{proof}
Note that if $X$ has a dense set of isolated points then $X^\N$ has a dense set of $1$-point components. By Theorem~\ref{medini}, the space $X^\N$ is $h$-homogeneous, and since $X^\N$ is first-countable, the result now follows from Theorem~\ref{stronghomogeneousnonrec}.
\end{proof}

In other words, for every first-countable compactification $\gamma \N$ of the countable discrete space, the product $(\gamma \N)^\N$ is non-reconstructible. We refer the reader back to Theorem~\ref{Tkaschukstheorem}, where such compactifications are constructed.

\begin{mylem}\label{C1dense_butrecon}
There is a compact metrizable space $X$ such that $C_1(X)$ is a dense co-dense $G_\delta$ but $X$ \emph{is} reconstructible. 
\end{mylem}
\begin{proof}
Let $Y$ and $Z$ be non-homeomorphic compact metrizable, $h$-homogeneous non-reconstructible spaces, as given by Corollary~\ref{firstctblehhom}, different from the Cantor set. Let $X=Y \oplus Z$. Then $X$ is a compact metrizable space, and $C_1(X)$ is a dense co-dense $G_\delta$. However, $X$ does not have a universal sequence, because the sequence would need to be constant of type $\sequence{Y}$ at points of $Y$, but constant of type $\sequence{Z}$ at points of $Z$. Hence, $X$ is reconstructible.
\end{proof}

\newcommand{\block}{block}

\subsection{A geometric construction of spaces with dense and co-dense $C_1$}
\label{section253}

We now present machinery for geometric constructions of compact metrizable non-reconstructible spaces. Our construction takes place in the unit cube $I^3$ and works by constructing a decreasing sequence of compact subsets of $I^3$.

\subsubsection*{An informal description}
The basic building block of our construction consists of a planar continuum and a countable sequence of cubes of exponentially decreasing diameter \lq approaching\rq\ this continuum from above so that every point of the continuum is a limit of cubes. In other words, our basic building block is a compactification of $I^3 \times \N$  embedded in $I^3$ with remainder a planar continuum. 
If the planar continuum is $E$, we will call a space of this type a \block\ with basis $E$.

Let $X_1$ be such a \block\ with basis $E_\emptyset$. To obtain $X_2$ we replace the $k$th cube, denoted by $F_{\sequence{k}}$, by an appropriately scaled \block\ with basis $E_{\sequence{k}}$, and we do this for each $k$ (see Figure \ref{figuregeometric}). In $X_2$ we index the new cubes by elements of $\N^2$. Clearly, $X_2 \subset X_1$. We repeat this procedure to inductively construct the $X_n$.

\begin{center}
\begin{figure}[ht]
\begin{tikzpicture}
\draw (0,8) rectangle (8,0); 
\draw[thick] (1,0.75) -- (7,0.75); 

\draw (1,7) rectangle (3,5); 
\draw (4,4.5) rectangle (5,3.5); 
\draw (5.5,3) rectangle (6,2.5); 
\draw (1.25,2.25) rectangle (1.5,2); 
\draw (3.5,1.875) rectangle (3.675,1.75); 
\draw[dotted] (4.25,1.5) -- (6.5,1);

\node[below left] at (8,8) {$I^3$};

\node[right] at (3,7) {$F_{\sequence{1}} = F_1^{E_\emptyset,D_\emptyset}$};
\node[right] at (5,4.5) {$F_{\sequence{2}} = F_2^{E_\emptyset,D_\emptyset}$};
\node[right] at (6,3) {$F_{\sequence{3}}$};
\node[right] at (1.5,2.25) {$F_{\sequence{4}}$};
\node[right] at (3.675,1.875) {$F_{\sequence{5}}$};

\node[right] at (7,0.75) {$E_\emptyset$};

\draw[fill=blue] (2,0.75) circle (0.075);
\node[anchor=north] at (2,0.75) {$d_{\emptyset,1}$};

\draw[fill=blue] (4.5,0.75) circle (0.075);
\node[anchor=north] at (4.5,0.75) {$d_{\emptyset,2}$};

\draw[fill=blue] (5.75,0.75) circle (0.075);
\node[anchor=north] at (5.75,0.75) {$d_{\emptyset,3}$};

\draw[fill=blue] (1.375,0.75) circle (0.075);
\node[anchor=north] at (1.375,0.75) {$d_{\emptyset,4}$};

\draw[fill=blue] (3.5875,0.75) circle (0.075);
\node[anchor=north] at (3.5875,0.75) {$d_{\emptyset,5}$};

\begin{scope}[shift={(1,5)},scale=0.25]
\draw[thick] (1,0.75) -- (5.75,0.75); 

\draw (1,7) rectangle (3,5); 
\draw (4,4.5) rectangle (5,3.5); 
\draw (5.5,3) rectangle (6,2.5); 
\draw (1.25,2.25) rectangle (1.5,2); 
\draw (3.5,1.875) rectangle (3.675,1.75); 
\draw[dotted] (4.25,1.5) -- (5.5,1);

\end{scope}
\begin{scope}[shift={(1,5)}, scale=0.25, every node/.style={scale=0.75}]
\node[right] at (3,7) {$F_{\sequence{1,1}}$};
\node[right] at (5,4.5) {$F_{\sequence{1,2}}$};

\node[right] at (5.75,0.75) {$E_{\sequence{1}}$};
\end{scope}

\begin{scope}[shift={(4,3.5)},scale=0.125]
\draw (1,0.75) -- (7,0.75); 

\draw (1,7) rectangle (3,5); 
\draw (4,4.5) rectangle (5,3.5); 
\draw (5.5,3) rectangle (6,2.5); 
\draw (1.25,2.25) rectangle (1.5,2); 
\draw (3.5,1.875) rectangle (3.675,1.75); 
\draw[dotted] (4.25,1.5) -- (6.5,1);

\end{scope}

\begin{scope}[shift={(5.5,2.5)},scale=0.0625]
\draw (1,0.75) -- (7,0.75); 

\draw (1,7) rectangle (3,5); 
\draw (4,4.5) rectangle (5,3.5); 
\draw (5.5,3) rectangle (6,2.5); 
\draw (1.25,2.25) rectangle (1.5,2); 
\draw (3.5,1.875) rectangle (3.675,1.75); 
\draw[dotted] (4.25,1.5) -- (6.5,1);

\end{scope}

\begin{scope}[shift={(1.25,2)},scale=0.03125]
\draw[thin] (1,1.75) -- (6.5,1.75); 

\draw[thin] (1.5,6) rectangle (3.5,4); 
\draw[dotted] (4,4) -- (6.5,1);

\end{scope}

\end{tikzpicture}
\setlength{\belowcaptionskip}{-10pt}
\caption{A sketch of $X_2$}\label{figuregeometric}
\end{figure}
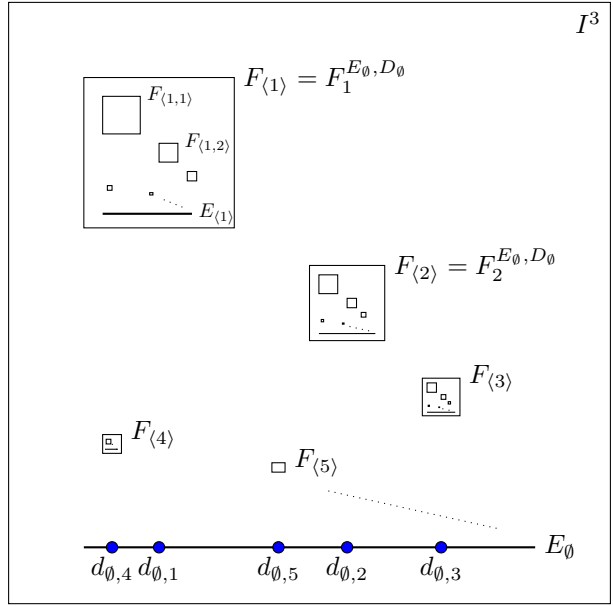
\end{center}


\subsubsection*{Relativizing to cubes}
For a closed cube $C=\closure{B}^{\infty}_r(c) = \set{x \in \R^3}:{d_\infty(x,c)\le r}$ with center $c$ and side-length $2r$ in $\R^3$ we let $a_C:I^3 \to C$ be the natural affine map given by $x \mapsto 2r(x-(\frac12,\frac12,\frac12))+c$. If $A$ is a subset of $I^3$ we say that its image under $a_C$ is the relativization of $A$ to $C$ and write $\relativization{A}{C}$ for this image.

\subsubsection*{The basic building block}
Given a planar continuum $E \subset \left[\frac14,\frac34\right]^2 \times \singleton{\frac12} \subset I^3$ and a countable dense enumerated subset $D=\set{d_k}:{k \in \N}$ of $E$, we define $$
F_{k}^{E,D} = \closure{B}^\infty_{2^{-2k-1}}\of{d_k+\p{0,0,2^{-2k}}} 
$$ 
and 
$$C_{E,D} = E \union \Union_{k \in \N} F_{k}^{E,D} .$$
We have that $C_{E,D} \subset I^3$, and for distinct $k,l \in \N$ that $F_{k}^{E,D} \cap F_{l}^{E,D}=\emptyset$.

\subsubsection*{The recursive construction}
Our construction uses as input a countable list $\script{E}=\set{(E_f,D_f)}:{f\in \N^{<\N}}$ (where $\N = \Set{1,2,\ldots}$ and $\N^{<\N}=\bigcup_{n \in \N} \N^n$) of non-trivial planar continua $E_f \subset \left[\frac14,\frac34\right]^2 \times \singleton{\frac12} \subset I^3$, and countable dense subsets $D_f \subset E_f$ with a fixed enumeration $D_f = \set{d_{f,k}}:{k \in \N}$. 

From this list $\script{E}$, we will build a decreasing sequence of compact metrizable spaces $X_n \subset I^3$, each with a designated collection of closed cubes $\set{F_f}:{f \in \N^n}$.

We will start by defining
$$
X_1 	= \relativization{C_{E_\sequence{\emptyset},D_\sequence{\emptyset}}}{I^3} \; \textnormal{ and, for all $k \in \N$, } \; 
F_{\sequence{k}} = \relativization{F_{k}^{E_\sequence{\emptyset},D_\sequence{\emptyset}}}{I^3}.
$$
The relativization at this point is trivial and has only been included for the sake of clarity.
Having defined $X_n$ and $ \set{F_f}:{f \in \N^n}$ for some $n \in \N$, we set
%
$$
X_{n+1} = \left( X_n \setminus \Union_{f \in \N^n} F_f \right) \union \Union_{f \in \N^n} \relativization{C_{E_f,D_f}}{F_f} 
$$
and for $k \in \N$ and $f \in \N^n$
$$
F_{f\concat k} = \relativization{F_{k}^{E_f,D_f}}{F_f}
	.
$$
Since the $X_n$ form a decreasing sequence of non-empty compact sets, the space $X_\script{E} = \Intersection_n X_n$ is a non-empty compact metric space.

\subsubsection*{Key properties of the construction} The next lemma gathers the necessary particulars about our construction, and can be easily verified.
 
\begin{mylem}
\label{constrlemm1}
Let $\script{E}=\set{(E_f,D_f)}:{f\in \N^{<\N}}$ be a list of non-trivial planar continua, and consider the compact metrizable space $X_\script{E}$ as described above. 
\begin{enumerate}
\itemsep0em 
\item For all $g \in \N^{<\N}$ we have $F_g \cap X_\script{E}$ is homeomorphic to $X_{\tilde{\script{E}}}$ where $\tilde{\script{E}}=\set{(\tilde{E}_f, \tilde{D}_f)}:{f\in \N^{<\N}}$ with $\tilde{E}_f = E_{g \concat f}$ and $\tilde{D}_f = D_{g \concat f}$.
\item The union over $\script{F}_n = \set{F_f \cap X_\script{E}}:{f \in \N^n}$ equals $X_\script{E} \setminus X_{n-1}$ and is a dense open subspace of $X_\script{E}$.
\item The collection $\script{F}=\bigcup_n \script{F}_n$ forms a clopen $\pi$-basis for $X_\script{E}$.
\item We have $C_1(X_\script{E}) = X_\script{E} \setminus \bigcup_{f \in \N^{<\N}} \relativization{E_f}{F_f}$ is a dense, co-dense $G_\delta$ in $X_\script{E}$. 
\end{enumerate}
\end{mylem}


\subsection{More non-reconstructible $h$-homogeneous spaces}
\label{geometricex1}
Applying the above construction in the special case that we choose all $E_f$ to be identical, we get a non-reconstructible $h$-homogeneous space. Of course, taking all the $E_f$ to be a single point, we obtain the Cantor set which is the inspiration for the above construction.

\begin{mythm}
\label{constructionhhomo1}
Suppose $\script{E}^{(0)} = \singleton{(E,D)}$ for some planar continuum $E$ and a fixed enumerated countable dense set $D \subset E$. Then the space $X_{\script{E}^{(0)}}$ is a non-reconstructible compact metrizable $h$-homogeneous space.
\end{mythm}

For the proof we need the following lemma.

\begin{mylem}[Matveev {\cite{matveev}}]
\label{matveevslem}
Assume that $X$ has a $\pi$-base consisting of clopen sets that are homeomorphic to $X$. If there exists a sequence $\Sequence{U_n}:{n \in \w}$ of non-empty open subsets of $X$ converging to a point of $X$ then $X$ is $h$-homogeneous. \qed
\end{mylem}

For a nice proof see also \cite[Appendix A, Prop.\ 17]{medini}. 

\begin{proof}[Proof of Theorem~\ref{constructionhhomo1}]
The sets $F_f$ for $f \in \N^{<\N}$ form a clopen $\pi$-basis, all elements of which are homeomorphic to $X_{\script{E}^{(0)}}$ by Lemma~\ref{constrlemm1}(1). Moreover, $X_{\script{E}^{(0)}}$ has a dense set of one-point components by Lemma~\ref{constrlemm1}(4). Thus, $X_{\script{E}^{(0)}}$ is $h$-homogeneous by Lemma~\ref{matveevslem}, and it follows from Theorem~\ref{stronghomogeneousnonrec} that it is non-reconstructible.
\end{proof}

\subsection{Non-reconstructible spaces without a constant universal sequence}
\label{geometricex2}

All our examples of non-reconstructible spaces so far had a constant universal sequence. In this section, we answer our third question and show that this need not be the case: We use the construction presented in Section~\ref{section253} to build two non-reconstructible compact metrizable spaces without a constant universal sequence.

Despite not having a constant universal sequence, our first example is sufficiently self-similar so that non-reconstructibility can be verified directly. The second construction is much more subtle since it has a lot of rigidity built in. Analysing these two spaces in detail led the authors to the main characterisation in Theorem~\ref{recchartthm}.

\subsubsection*{A non-reconstructible space without constant universal sequence I}
For the first space, let $\set{E_n}:{n \in \w}$ be a list of pairwise non-homeomorphic planar continua, and consider the list $\script{E}^{(1)}=\set{(E_f,D_f)}:{f\in \N^{<\N}}$ such that for $f \colon k \to \N$ we have $E_f = E_{f(k-1)}$ and  $D_f = D_{f(k-1)}$. Consider the space $X_\script{E}$ as described in Section~\ref{section253}. Our list ensures that at the $n$th step in the recursion we replace the cubes always by the same building block $C_{E_n,D_n}$. Thus the homeomorphism type of $F_f \intersect X_\script{E}$ only depends on the length of $f \in \N^{<\omega}$ and we write $T_n$ for the homeomorphism type of $F_f \intersect X_{\script{E}}$ for any $f \in \N^n$.

\begin{mythm}\label{noconseq1}
\label{existuniversalsequence}
The space $X_{\script{E}^{(1)}}$ is an example of a non-reconstructible compact metrizable space without a constant universal sequence.
\end{mythm}

\begin{proof}
Since $T_{n+1}$ is a clopen subset of $T_n$, every point $x\in  \relativization{E_f}{F_f} $ is seen to be the limit of a sequence homeomorphic to some tail of $\Sequence{T_n}:{n \in \N}$, living in $F_f$. By density (Lemma~\ref{constrlemm1}(4)), every point of $X_{\script{E}^{(1)}}$ is the limit of a sequence homeomorphic to some tail of $\Sequence{T_n}:{n \in \N}$. 
As in the Universal Sequence Proposition~\ref{universalsequence}, this implies that $X_{\script{E}^{(1)}}$ has a universal sequence of type $\Sequence{T_n}:{n > N}$. Hence, this space is non-reconstructible.

To see that $X_{\script{E}^{(1)}}$ does not contain a constant universal sequence, we argue it is not $\pi$-homogeneous (Lemma~\ref{pihomsequences}). Indeed, any clopen subset $B  \subset X_{\script{E}^{(1)}}$ contains a non-trivial component, say $E_n$. But $T_k$ for $k > n$ does not contain a copy of $E_n$, and hence subsets homeomorphic to $B$ cannot form a $\pi$-base for $X_{\script{E}^{(1)}}$.
\end{proof}

One can also verify directly that $X_{\script{E}^{(1)}} \oplus \bigoplus_{n \in \N} T_n$ is a reconstruction of $X_{\script{E}^{(1)}}$.

\subsubsection*{A non-reconstructible space without constant universal sequence II}

For our last example, we again use the construction from Section~\ref{section253}. For our list $\script{E}^{(2)}=(E_f,D_f)$ we choose the $E_f$ to be pairwise non-homeomorphic planar continua such that each continuum appears only finitely often. So let $E_n$ be countably many distinct planar continua  embedded in $\left[\frac14,\frac34\right]^2 \times \singleton{\frac12}$ and $D_n \subset E_n$ enumerated countable dense subsets of $E_n$. For each $n$, partition $D_n$ into countably many finite consecutive subsequences $D_m^n=\set{d_k}:{k = N_m,\dots,N_{m+1}-1}$ such that $D_m^n$ is $2^{-m}$-dense in $E_n$.

We now let $E_{\sequence{\emptyset}} = E_0$, $E_{\sequence{n}} = E_k$ if $n \in D_k^0$. We then inductively define $E_{f\concat n} = E_k $ if $n \in D_k^m$ where $m$ is such that $E_f = E_m$.

Our choices above ensure that for any two $F_f,F_g$ for which $E_f = E_g$, the space $X \intersect F_f$ and $X \intersect F_g$ are homeomorphic (via the relativization homeomorphisms). We will thus define $X_n = X \intersect F_f$ where $n$ is such that $E_f = E_n$. 

\begin{mythm}\label{noconseq2}
\label{existuniversalsequence}
The space $X_{\script{E}^{(2)}}$ is an example of a non-reconstructible compact metrizable space in which every non-trivial component appears only finitely often.
\end{mythm}

\begin{proof}
By the partitioning of $D_n$, every point $x\in  \relativization{E_f}{F_f} $ is the limit of a sequence homeomorphic to some tail of $\Sequence{X_n}:{ n>N}$. Now proceed as above.
\end{proof}

 If we choose pairwise incomparable $E_n$ (no continuous map of one onto another, \cite{Wara32}), then every homeomorphism of $X_\script{E}$ induces a permutation of the $E_f$.

%
%
%
%

	  
\end{document}